\documentclass[11pt]{article}
\title{Rainbow matchings and algebras of sets\footnote{An extended abstract of this paper appeared in \emph{Eurocomb 2015} (\emph{Electronic Notes in Discrete Mathematics} 49:251--257, 2015).}}
\date{}
\author{Gabriel Nivasch\footnote{\texttt{gabrieln@ariel.ac.il}. Department of Computer Science, Ariel University, Ariel, Israel.} \and Eran Omri\footnote{\texttt{omrier@ariel.ac.il}. Department of Computer Science, Ariel University, Ariel, Israel.}}

\usepackage{amsthm, amsmath, amsfonts}
\usepackage{graphicx}

\newcommand{\eqv}[1]{\mathrel{\sim_{#1}}}
\newcommand{\Cleft}{C_{\mathrm{left}}}
\newcommand{\Cright}{C_{\mathrm{right}}}

\newtheorem{theorem}{Theorem}
\newtheorem{lemma}[theorem]{Lemma}
\newtheorem{observation}[theorem]{Observation}

\theoremstyle{definition}
\newtheorem{charging}{Charging Scheme}

\theoremstyle{remark}
\newtheorem{remark}{Remark}

\begin{document}
\maketitle

\begin{abstract}
Grinblat (2002) asks the following question in the context of algebras of sets: What is the smallest number $\mathfrak v = \mathfrak v(n)$ such that, if $A_1, \ldots, A_n$ are $n$ equivalence relations on a common finite ground set $X$, such that for each $i$ there are at least $\mathfrak v$ elements of $X$ that belong to $A_i$-equivalence classes of size larger than $1$, then $X$ has a rainbow matching---a set of $2n$ distinct elements $a_1, b_1, \ldots, a_n, b_n$, such that $a_i$ is $A_i$-equivalent to $b_i$ for each $i$?

Grinblat has shown that $\mathfrak v(n) \le 10n/3 + O(\sqrt{n})$. He asks whether $\mathfrak v(n) = 3n-2$ for all $n\ge 4$. In this paper we improve the upper bound (for all large enough $n$) to $\mathfrak v(n) \le 16n/5 + O(1)$.
\end{abstract}

\section{Introduction}
In this paper we attack a combinatorial problem previously raised and studied by Grinblat in~\cite{g02,g04,g_new}. Grinblat's question arose in the study of algebras of sets. Given a nonempty ground set $Y$, let $\mathcal P(Y)$ denote the power set of $Y$. An \emph{algebra (of sets)} on $Y$ is a nonempty family $A\subseteq \mathcal P(Y)$ such that: (1) if $M\in A$ then also $Y\setminus M \in A$; and (2) if $M_1,M_2\in A$ then also $M_1\cup M_2\in A$. Grinblat investigated necessary and sufficient conditions under which the union of at most countably many algebras on $Y$ equals $\mathcal P(Y)$.

In this context, Grinblat asks for the smallest integer $\mathfrak v = \mathfrak v(n)$ for which the following holds: ``Suppose $A_1, \ldots, A_n$ are $n$ algebras on $Y$, such that, for each $i$, there exists a collection of at least $\mathfrak v$ pairwise disjoint subsets of $Y$ that are \emph{not} in $A_i$. Then there exists a family $\{U_1, V_1, \ldots, U_n, V_n\}$ of $2n$ pairwise-disjoint subsets of $Y$ such that, for each $i$ and each $Q\subseteq Y$, if $Q$ contains one of $U_i$, $V_i$ and is disjoint from the other, then $Q\notin A_i$.''

Grinblat shows that this problem is equivalent to the combinatorial problem presented below. (The equivalence is quite straightforward if the ground set $Y$ is finite, but if $Y$ is infinite then the argument is more delicate. We refer the reader to~\cite{g_new} for more details.)

\subsection{The combinatorial problem}

The combinatorial problem that interests us is the following: Let $n$ be a positive integer. Let $X$ be a finite ground set, and let $A_1, \ldots, A_n$ be $n$ equivalence relations on $X$ (or equivalently, partitions of $X$ into subsets). If $a, b\in X$ are equivalent under $A_i$, then we say for short that $a, b$ are $i$-equivalent, and we write $a \eqv{i} b$. The $i$-equivalence class of an element $a\in X$ is given by $[a]_i = \{ b\in X : a \eqv{i} b\}$. The \emph{kernel} of $A_i$, denoted $K_i$, is defined as the set of elements of $X$ that are $i$-equivalent to some element other than themselves:
\begin{equation*}
K_i = \bigl\{a\in X : \bigl|[a]_i\bigr| > 1 \bigr\}.
\end{equation*}

(It will become evident that one can assume without loss of generality that all equivalence classes in each $A_i$ have size at most $3$.)

We shall call a set of $2n$ \emph{distinct} elements $a_1, b_1, \ldots, a_n, b_n \in X$ a \emph{rainbow matching} if $a_i \eqv{i} b_i$ for each $i$. (See e.g.~Glebov et al.~\cite{GSS} for the term.)

The problem is to find the smallest integer $\mathfrak v = \mathfrak v(n)$ such that, if $|K_i|\ge \mathfrak v$ for all $i$, then $A_1, \ldots, A_n$ have a rainbow matching.

Grinblat observed that $\mathfrak v(n) \ge 3n-2$: If we let all equivalence relations $A_i$ be identical, consisting of $n-1$ equivalence classes of size $3$, then they have no rainbow matching even though $|K_i| = 3n-3$.

Grinblat also showed that $\mathfrak v(3) = 9$. The lower bound $\mathfrak v(3) > 8$ is illustrated in Figure~\ref{fig_n_3}.

Grinblat proved in~\cite{g_new} that $\mathfrak v(n) \le \left\lceil 10n/3 +\sqrt{2n/3}\right\rceil$ (he previously announced a slightly weaker bound in~\cite{g04}). He asks whether $\mathfrak v(n) = 3n-2$ for all $n\ge 4$.

In this paper we improve the upper bound to $\mathfrak v(n) \le 16n/5 + O(1)$:

\begin{theorem}\label{thm_bd}
Let $A_1, \ldots, A_n$ be $n$ equivalence relations with kernels $K_1, \ldots, K_n$, respectively. Suppose $|K_i| \ge (3+1/5)n+c$ for each $i$, where $c$ is a large enough constant. Then $A_1, \ldots, A_n$ have a rainbow matching.
\end{theorem}

\begin{figure}
\centerline{\includegraphics{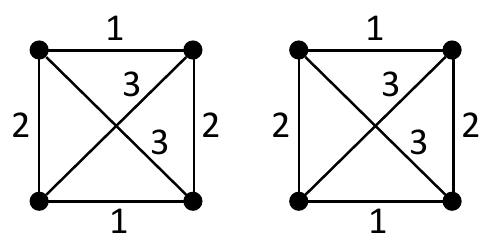}}
\caption{\label{fig_n_3}Here $|K_i| = 8$ for all $i=1,2,3$, and yet there is no rainbow matching.}
\end{figure}

As we will see, it is enough to take $c=5000$ in Theorem~\ref{thm_bd}.

\subsection{Overview of the proof}

We prove Theorem~\ref{thm_bd} by a modification of Grinblat's argument. The proof follows by induction on the number of equivalence relations $n$, showing that given a rainbow matching (of $n-1$ pairs) for the equivalence relations $A_2,\ldots,A_n$, it is possible to obtain a rainbow matching for $A_1, A_2,\ldots,A_n$.
The proof follows in three main steps, where in each step, we observe that it is either possible to complete a full fledged rainbow matching, or to slightly extend the previous formation at hand. The final formation, which is the result of the third step, allows us to complete the rainbow matching, concluding the proof.

As mentioned above, we start with a rainbow matching over $A_2,\ldots,A_n$ of the form $\{a_i,b_i\}_{i = 2}^{n}$, where $a_i \eqv{i} b_i$. 
In the first step, we transform the first  $t-1 = \lfloor n/5 \rfloor$ pairs into a \emph{track} of $t-1$ \emph{components} of the form $\{a_i,b_i,c_i,d_i\}_{i = 2}^{t}$, where $a_i \eqv{i-1} c_i$ and $b_i \eqv{i-1} d_i$ (the choice of indices is without loss of generality, by some renaming of indices). This $(t-1)$-long track is called $\Cleft$ and the remaining $4n/5$ pairs are called $\Cright$ (see Figure~\ref{fig_train}). 
The construction of $\Cleft$ is inductive, and follows by showing that an $m$-long track, for $m\le t-2$, can be extended.  This is true, since otherwise, there exists a pair $x \eqv{m+1} y$ that allows us to complete a full fledged rainbow matching. See Lemma~\ref{lemma_train} below for the complete argument.  

In the second step, we consider the elements of $K_1$ and $K_{t}$. Specifically, we count the number of such elements that are connected to a component $C_i$ in $\Cleft$ or to a component $C_j$ in $\Cright$. We first observe that no such $C_i$ or $C_j$ can account for more than four elements of either $K_1$ and $K_{t}$. Otherwise, it is not hard to complete a full fledged matching. We call a component \emph{heavy} if it accounts for at least 7 elements in $K_1\cup K_{t}$. We further observe that the existence of $5$ heavy components in $\Cleft$ enables us to complete the desired matching. We, hence, assume that all but four heavy components appear in $\Cright$. Let $H$ be the set of heavy components in $\Cright$. A simple counting argument yields that $H$ is of size at least  $n/5+c-4$. 
Assuming that we cannot complete a rainbow matching using a single component $C_i$ with $i\in H$, we move on to the following third step. 

The third and final step pinpoints a component $C_{j^*} = \{a_{j^*},b_{j^*}\}$ in $\Cright$, such that there exist $u_i,v_i\in K_i$ with $a_{j^*}\eqv{i}u_i$ and $b_{j^*}\eqv{i}v_i$ for many $i\in H$. Furthermore, we prove the existence of two indices $i_1,i_2\in H$, and a ``free'' pair of elements $x\eqv{j^*}y$, such that the following substitutions are possible, completing a full fledged rainbow matching: $x\eqv{j^*}y$ will represent $A_{j^*}$ (replacing $C_{j^*}$), the two pairs representing $A_{i_1}, A_{i_2}$ will be replaced with, say, $a_{j^*}\eqv{i_1}u_{i_1}$ and $b_{j^*}\eqv{i_2}v_{i_2}$, and finally, $A_1$ and $A_t$ will now be represented by, say, $a_{i_1}\eqv{1}q$ and $a_{i_2}\eqv{t}p$, where $q,p$ exist since $i_1,i_2\in H$ (see Figure~\ref{fig_jstar} for an illustration). 

\subsection{Subsequent work}

In a follow-up paper, Clemens et al.~\cite{CEP} have solved the problem asymptotically, by showing that $\mathfrak v(n) \le 3n + O(\sqrt n)$.

\section{Proof of Theorem~\ref{thm_bd}}\label{sec:mainProof}

Suppose that $|K_i| \ge (3+1/5)n+c$ for each $i = 1, \ldots, n$, for some constant $c$ to be specified later. We can assume by induction on $n$ that $A_2, \ldots, A_n$ have a rainbow matching $a_2 \eqv{2} b_2$, $a_3 \eqv{3} b_3$, $\ldots$, $a_n \eqv{n} b_n$. Let $B = \{a_2, b_2, \ldots, a_n, b_n\}$.

\begin{observation}\label{obs_1}
If there exist two distinct elements $a_1 \eqv{1} b_1$ with $a_1, b_1 \in X \setminus B$ then we are immediately done.
\end{observation}

Hence, let us assume that the above is not the case. Thus, every element in $K_1 \setminus B$ must be $1$-equivalent to some element of $B$ (possibly more than one). However, no two distinct elements of $K_1 \setminus B$ can be $1$-equivalent to the \emph{same} element of $B$ (by the transitivity of $\eqv{1}$).

Therefore, by the pigeonhole principle, there must exist an index\footnote{Actually, many.} in $\{2, \ldots, n\}$, which without loss of generality we assume to be $2$, for which there exist two distinct elements $c_2, d_2 \in X \setminus B$ satisfying $a_2 \eqv{1} c_2$, $b_2 \eqv{1} d_2$.

We can now similarly consider $K_2$: Unless we are immediately done, there must exist an index in $\{3, \ldots, n\}$, which without loss of generality we assume to be $3$, for which there exist two distinct elements $c_3, d_3 \in X \setminus (B\cup \{c_2, d_2\})$ satisfying $a_3 \eqv{2} c_3$, $b_3 \eqv{2} d_3$.

We can continue in this way:

\begin{lemma}\label{lemma_train}
Let $t = \lfloor n/5 \rfloor$. We can find $t-1$ distinct indices in $\{2, \ldots, n\}$, which without loss of generality we assume to be $2, \ldots, t$, and we can find $2(t-1)$ pairwise distinct elements $c_2, d_2, \ldots, c_t, d_t \in X \setminus B$, such that $a_i \eqv{i-1} c_i$ and $b_i \eqv{i-1} d_i$ for all $2\le i\le t$.
\end{lemma}

\begin{proof}
Suppose by induction that we have already found $c_2, d_2, \ldots, c_i, d_i$.

Let $B'= B \cup \{c_2, d_2, \ldots, c_i, d_i\}$. Partition $B'$ into ``components" as follows: $C_2 = \{a_2, b_2, c_2, d_2\}$, $\ldots$, $C_i = \{a_i, b_i, c_i, d_i\}$; $C_{i+1} = \{a_{i+1}, b_{i+1}\}$, $\ldots$, $C_n = \{a_n,b_n\}$. Let $\Cleft = C_2\cup \cdots\cup C_i$ and $\Cright = C_{i+1}\cup \cdots\cup C_n$. See Figure~\ref{fig_train}.

\begin{observation}\label{obs_unless}
If there exist two distinct elements $x \eqv{i} y$, with $x, y \in K_{i}\setminus \Cright$, then we are easily done \emph{unless} one of $x, y$ belongs to $\{a_j, c_j\}$ and the other one belongs to $\{b_j, d_j\}$ for the \emph{same} index $2\le j\le i$. See Figure~\ref{fig_unless}.
\end{observation}

\begin{figure}
\centerline{\includegraphics{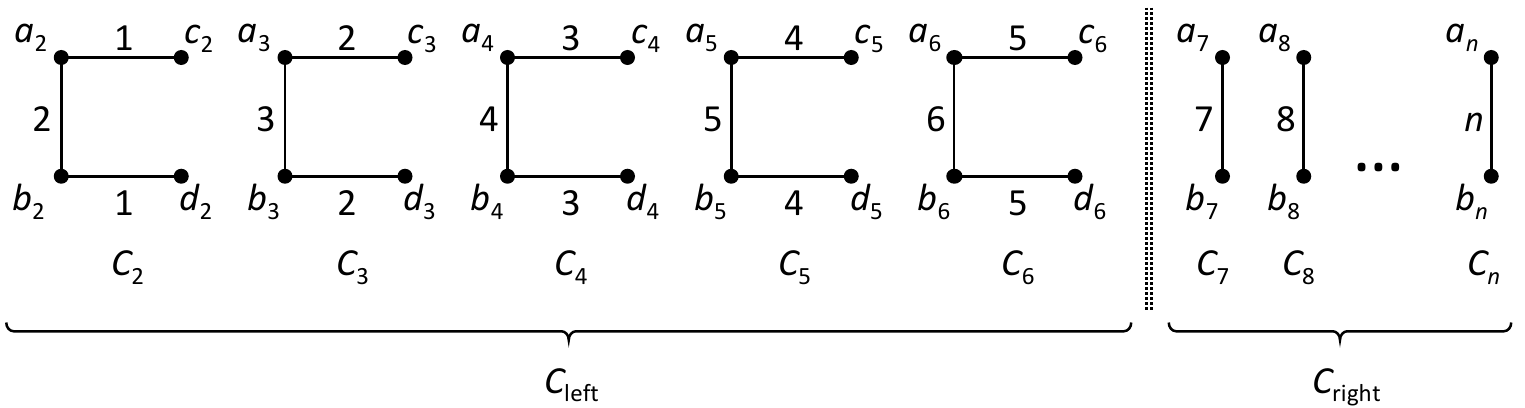}}
\caption{\label{fig_train}Left-side and right-side components.}
\end{figure}

\begin{figure}
\centerline{\includegraphics{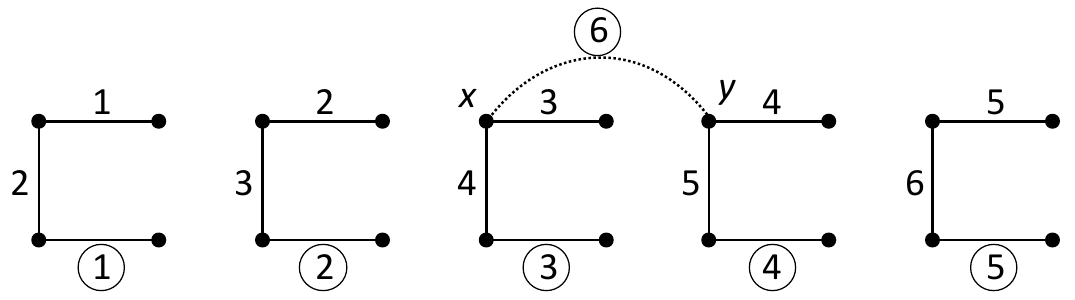}}
\caption{\label{fig_unless}If $x\eqv{i} y$ (here $i=6$), and $x,y$ belong to different left-side components, or they both belong to the top (or the bottom) part of the same left-side component (case not shown), then we can easily complete a rainbow matching (indicated by the circled edge labels).}
\end{figure}

Hence, let us charge each element of $K_{i}$ to exactly one component, as follows:

\begin{charging}
Let $x\in K_{i}$. If $x \in B'$, then $x$ is charged to the component it belongs to. Otherwise, by Observation~\ref{obs_unless}, $x$ must be $i$-equivalent to some $y\in \Cright$; then we charge $x$ to $y$'s component. (If $x$ can be charged to more than one right-side component, then we choose one of them arbitrarily.)
\end{charging}

The total number of charges is equal to $|K_{i}|$, which is at least $(3 + 1/5)n+c$. By Observation~\ref{obs_unless} and the transitivity of $\eqv{i}$, no component can get more than four charges. Hence, if $i\le n/5$, then there must be a component in $\Cright$ that received four charges. Without loss of generality it is $C_{i+1}$. Of the four elements charged to it, the two not belonging to it are the desired $c_{i+1}, d_{i+1}$. 

This concludes the proof of Lemma~\ref{lemma_train}.
\end{proof}

Define the set $B'$, the components $C_2, \ldots, C_n$, and the sets $\Cleft$ and $\Cright$ as above, with $t$ in place of $i$. Hence, $\Cleft= C_2\cup \cdots\cup C_t$ and $\Cright = C_{t+1}\cup \cdots\cup C_n$.

We now use the following charging scheme for $A_1$ and $A_t$:

\begin{charging}
Consider an element $x\in K_1$. If $x\in B$, then we $1$-charge $x$ to the component it belongs to. Otherwise, by Observation~\ref{obs_1}, $x$ must be $1$-equivalent to some element $y\in B$; then we $1$-charge $x$ to the component that contains $y$.

Consider the elements of $K_t$. We $t$-charge every element $z\in (K_t \cap B)$ to the component it belongs to. If $c_i \eqv{t} d_i$ for some $i$, then we $t$-charge both elements to the component $C_i$ that contains them. For every $z\in K_{t}$ not covered by the above cases, by Observation~\ref{obs_unless} there must be a component $C_i$ that contains an element $y\eqv{t}z$ (furthermore, either $C_i$ is a right-side component, or else $z\in C_i$); we charge $z$ to $C_i$.
\end{charging}

\begin{figure}
\centerline{\includegraphics{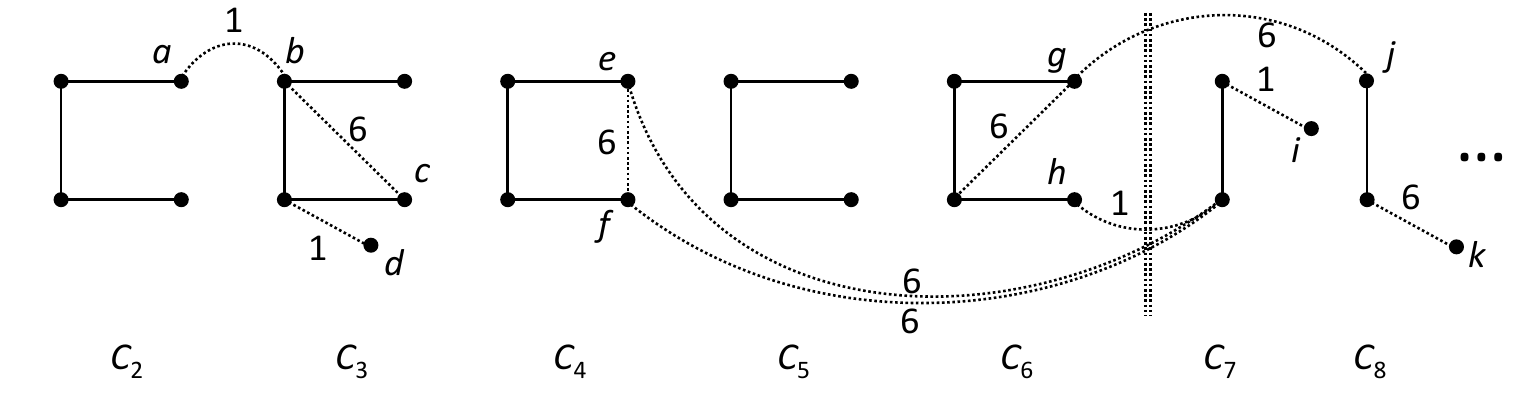}}
\caption{\label{fig_charging_2}Charging Scheme 2. Here $t=6$. Elements that belong to $B$ (such as $b$ and $j$) are always $1$- or $6$-charged to their own components. Elements $a$ and $d$ are $1$-charged to $C_3$. Elements $h$ and $i$ are $1$-charged to $C_7$. Element $c$ is $6$-charged to $C_3$. Elements $e$ and $f$ are $6$-charged to $C_4$ (despite the $6$-edges to $C_7$). Element $g$ is $6$-charged either to $C_6$ or to $C_8$. Element $k$ is $6$-charged to $C_8$.}
\end{figure}

Figure~\ref{fig_charging_2} illustrates Charging Scheme 2.

\begin{lemma}\label{lem_4_charges}
In Charging Scheme 2, no component $C_i$ can receive more than four $1$-charges, or more than four $t$-charges.
\end{lemma}

\begin{proof}
Since no two different elements outside $B$ can be $1$-equivalent to the same element of $B$, every component $C_i$ can receive at most two $1$-charges from its own elements $a_i$, $b_i$, plus at most two more $1$-charges from other elements.

The argument regarding $t$-charges is only slightly more complicated: There \emph{might} be an element $z$ in a right-side component $C_i$ that is $t$-equivalent to two different elements outside $B$. However, then they must be $c_j$, $d_j$ for some left-side component $C_j$, by Observation~\ref{obs_unless}. Hence, they are $t$-charged to $C_j$ and not to $C_i$.
\end{proof}

For each $2\le i \le n$, let $\sigma_i$ (resp.~$\tau_i$) be the number of $1$-charges (resp.~$t$-charges) that component $C_i$ received; let $S_i$ (resp.~$T_i$) be set of elements \emph{not} in $\{a_i,b_i\}$ that were $1$-charged (resp.~$t$-charged) to $C_i$; and let $U_i = S_i\cup T_i$.

The sets $S_i$ are by definition pairwise disjoint, as are the sets $T_i$. However, for a fixed~$i$, $S_i$ is not necessarily disjoint from $T_i$; and $U_i$ is not necessarily disjoint from $B'\setminus B$. Still, no three sets $U_i$ have a common intersection.

Lemma~\ref{lem_4_charges} states that $\sigma_i\le 4$ and $\tau_i\le 4$ for each $i$. Furthermore, by the argument in the proof of Lemma~\ref{lem_4_charges}, we have $|S_i|\le 2$ and $|T_i|\le 2$. Moreover, we have $\sum \sigma_i = |K_1|$ and $\sum \tau_i = |K_{t}|$, each of which is at least $(3+1/5)n+c$.

\begin{lemma}\label{lemma_5_7}
Suppose that there exist five different left-side components that receive at least $7$ charges each; namely, suppose there exist $C_{i_1}, \ldots, C_{i_5}$, with $2\le i_1< \cdots < i_5\le t$, such that $\sigma_{i_k} + \tau_{i_k} \ge 7$ for each $1\le k\le 5$. Then we can complete a rainbow matching.
\end{lemma}

\begin{proof}
Consider the component $C_{i_2}$. For simplicity rename its four elements $a', b', c', d'$ in the obvious way. Since this component received at least three $t$-charges, there must be a pair of elements among $a', b', c', d'$ that are $t$-equivalent \emph{excluding} the pair $\{a', b'\}$. This pair cannot be $\{a', c'\}$ nor $\{b', d'\}$, by Observation~\ref{obs_unless}. Hence, the pair must be $\{c', d'\}$ (case 1), or $\{a', d'\}$ or $\{b', c'\}$ (case 2).

In case 1, we consider components $C_{i_3}$, $C_{i_4}$, $C_{i_5}$. In each one of them there must be an $a_j$ or $b_j$ that is $1$-equivalent to some $y\notin B$. At most two of these $y$'s can be $c'$ or $d'$, so the third one leads to a win, as follows (see Figure~\ref{fig_win123}, top): Suppose for concreteness that $a_{i_5} \eqv{1} y$ for $y\notin\{c',d'\}$. Then we take the pairs $c'\eqv{t} d'$, $a_{i_5} \eqv{1} y$; the pairs $a_i \eqv{i} b_i$ for all $2\le i<i_5$; and the pairs $a_i \eqv{i-1} c_i$ for all $i_5 < i \le t$, except that if $y = c_i$ for some $i>i_5$ then we take $b_i\eqv{i-1} d_i$ instead. 

Now consider case 2. Suppose for concreteness that $a' \eqv{t} d'$ (the case $b' \eqv{t} c'$ is symmetric). Let us look again at component $C_{i_2}$. First suppose that it received four $1$-charges. Then each of $a', b'$ must be $1$-equivalent to an element not in $B$. Consider the element $y\eqv{1} b'$ that is not in $B$. If $y \neq d'$ then we go to case 2a below; if $y = d'$ then we go to case 2b below.

Now suppose $C_{i_2}$ received four $t$-charges. Then we must have \emph{both} $a' \eqv{t} d'$ and $b' \eqv{t} c'$. Furthermore, $C_{i_2}$ received at least three $1$-charges, so at least one of $a', b'$, say $b'$, must be $1$-equivalent to some $y\notin B$. As before, if $y \neq d'$ we go to case 2a; otherwise we go to case 2b.

Case 2a is an easy win by taking the pairs $b'\eqv{1} y$ and $a'\eqv{t} d'$, and completing the rainbow matching as in Figure~\ref{fig_win123} (middle).

\begin{figure}
\centerline{\includegraphics{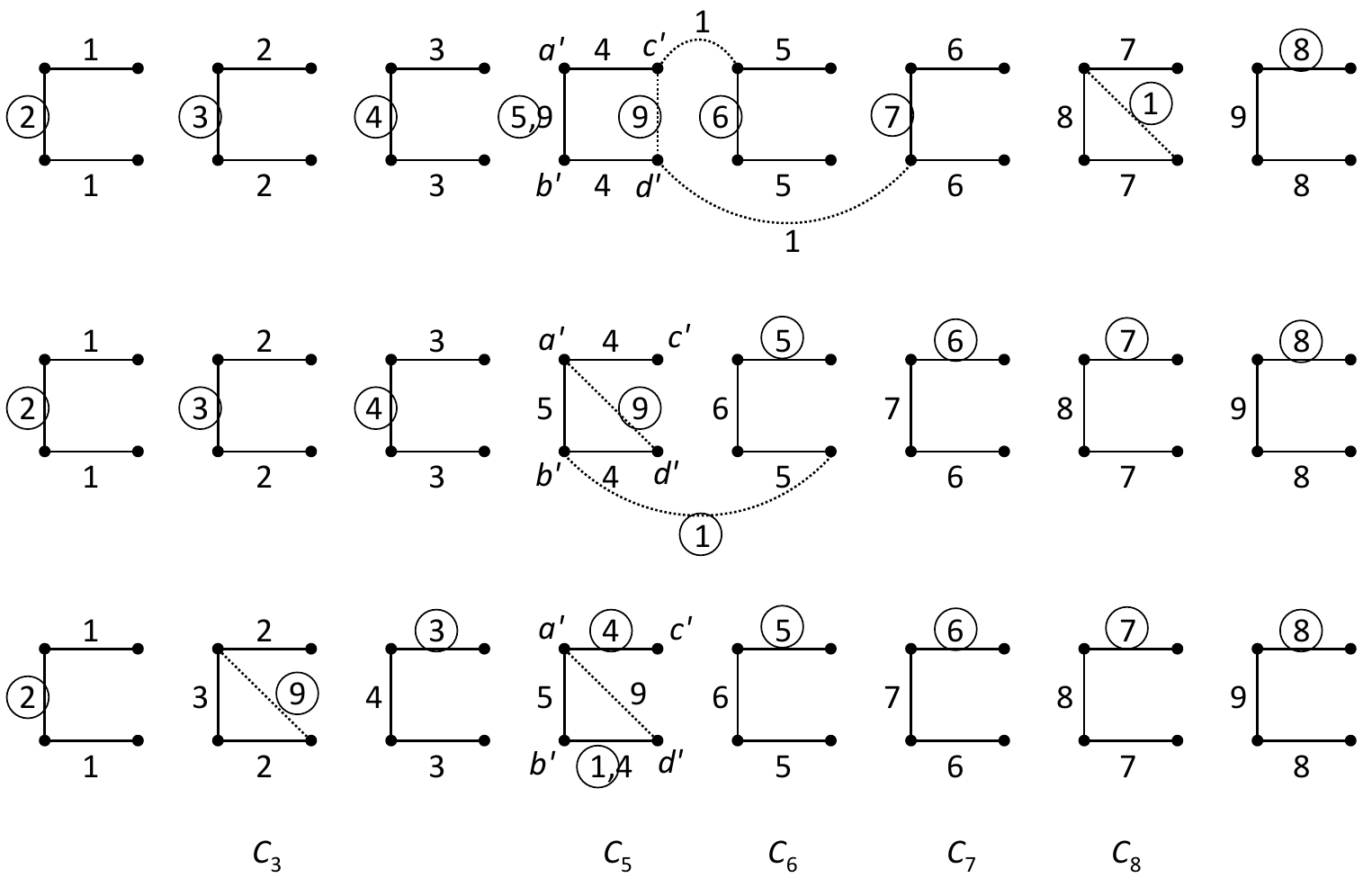}}
\caption{\label{fig_win123}Three cases considered in the proof of Lemma~\ref{lemma_5_7}. Here $t=9$, and the components that receive $7$ charges are $C_3$, $C_5$, $C_6$, $C_7$, $C_8$.}
\end{figure}

In case 2b, we take the pairs $a'\eqv{i_2-1} c'$ and $b'\eqv{1} d'$; from $C_{i_1}$ we take a pair $x\eqv{t} y$; and we complete the rainbow matching as in Figure~\ref{fig_win123} (bottom).
\end{proof}

Recall that no component can receive more than $8$ charges, and that the total number of charges is $2(16n/5+c)$. Therefore, the number of components that receive at least $7$ charges must be at least $n/5+c$. Lemma~\ref{lemma_5_7} implies that at most four of these components can be on the left side.

Hence, there must be at least $n/5+c-4\ge n/5$ right-side components that receive at least $7$ charges each. Call such components ``heavy", and let $H$ be the set of their indices; namely, let
\begin{equation*}
H = \bigl\{ i \in \{t+1, \ldots, n\} : \sigma_i + \tau_i \ge 7 \bigr\}.
\end{equation*}

For each $i\in H$ we have $|S_i|\ge 1$, $|T_i|\ge 1$, $\max{\{|S_i|,|T_i|\}}=2$, and $2\le |U_i|\le 4$. Furthermore, for every two distinct heavy indices $i,j\in H$ we have $|U_i\cup U_j| \ge 3$.

Let $i\in H$. It is not necessarily possible to find four distinct elements $v_1,v_2\in C_i$, $w_1,w_2\in U_i$, such that $v_1\eqv{1} w_1$ and $v_2\eqv{t} w_2$ (even if $\sigma_i = \tau_i = 4$) since we could have $a_i\eqv{1} x$, $b_i\eqv{1} y$, $a_i\eqv{t} y$, $b_i\eqv{t} x$. Nevertheless, we can prove the following lemma:

\begin{lemma}\label{lemma_2_heavy}
Let $C_i, C_j$, $i,j\in H$, be two distinct heavy components. Then we can find four distinct elements $v_1,v_2\in C_i\cup C_j$, $w_1,w_2\in U_i\cup U_j$, such that $v_1\eqv{1} w_1$ and $v_2\eqv{t} w_2$.

Furthermore, for any two fixed elements $q, r\in U_i\cup U_j$, it is always possible to do so guaranteeing that \emph{exactly~one} of $q$, $r$ belongs to $\{w_1,w_2\}$.
\end{lemma}

The ``furthermore" clause of Lemma~\ref{lemma_2_heavy} will be used once, in the proof of Lemma~\ref{lemma_jstar} below. Unfortunately, it requires a tedious case analysis.

\begin{proof}[Proof of Lemma~\ref{lemma_2_heavy}]
Suppose first that $q\in S_i \cap T_i$ (hence, $q\notin S_j$ and $q\notin T_j$, so $q\notin U_j$). We have $|U_j|\ge 2$, so there exists an element $s\in U_j\setminus\{r\}$. If $s\in S_j$ then we can take $w_1=s$, $w_2 = q$ and finish; otherwise, $s\in T_j$, so we are done by taking $w_1=q$, $w_2 = s$.

The case $q\in S_j\cap T_j$ is symmetric, as well as the cases $r\in S_i\cap T_i$ and $r\in S_j\cap T_j$. So suppose none of these cases apply.

Suppose for concreteness that $q\in S_i$ (the three other possibilities are symmetric). Consider $T_j$. If it contains an element $s\notin\{q,r\}$, then we are done by taking $w_1=q$, $w_2=s$. Hence, assume $T_j\subseteq\{q,r\}$.

Suppose $q\in T_j$. Suppose for concreteness that $q\eqv{1} a_i$ and $q\eqv{t} a_j$. Then $b_i$ must be $1$- or $t$-equivalent to an element $z_1\in U_i$, $z_1\neq q$; and $b_j$ must be $1$- or $t$-equivalent to an element $z_2\in U_j$, $z_2\neq q$. If one of $z_1, z_2$ is different from $r$, then we are done by taking it and $q$ for $\{w_1, w_2\}$. Otherwise, we have $r = z_1=z_2$. Take a third element $s\in (U_i\cup U_j) \setminus \{q,r\}$. Hence, $s$ is $1$- or $t$-equivalent to one of $a_i$, $a_j$, $b_i$, $b_j$. In the first two cases we take $\{w_1,w_2\} = \{r,s\}$, whereas in the last two cases we take $\{w_1, w_2\} = \{q,s\}$.

Finally, suppose $T_j = \{r\}$. Say $r$ is $t$-equivalent to $a_j$. Then $b_j$ must be $1$- or $t$-equivalent to some element $s\notin\{q,r\}$. Then we take $\{w_1,w_2\}$ to be $s$ and one of $q$, $r$.
\end{proof}

For $i\in H$, let us call a left-side component \emph{$i$-tainted} if it intersects $T_i$. Since $|T_i|\le 2$, for each $i$ there are at most two $i$-tainted components.

\begin{lemma}\label{lemma_outside}
Let $C_i$, $i\in H$ be a heavy component.
\begin{enumerate}
\item[(a)] If there exist two distinct elements $x\eqv{i} y$, both outside $B\cup S_i$, then we are done.

\item[(b)] Let $C_j$ be a left-side component that is not $i$-tainted. Then, if one of $a_j,b_j$ is $i$-equivalent to an element $z$ outside $B'\cup T_i$, then we are done.
\end{enumerate}
\end{lemma}

\begin{figure}
\centerline{\includegraphics{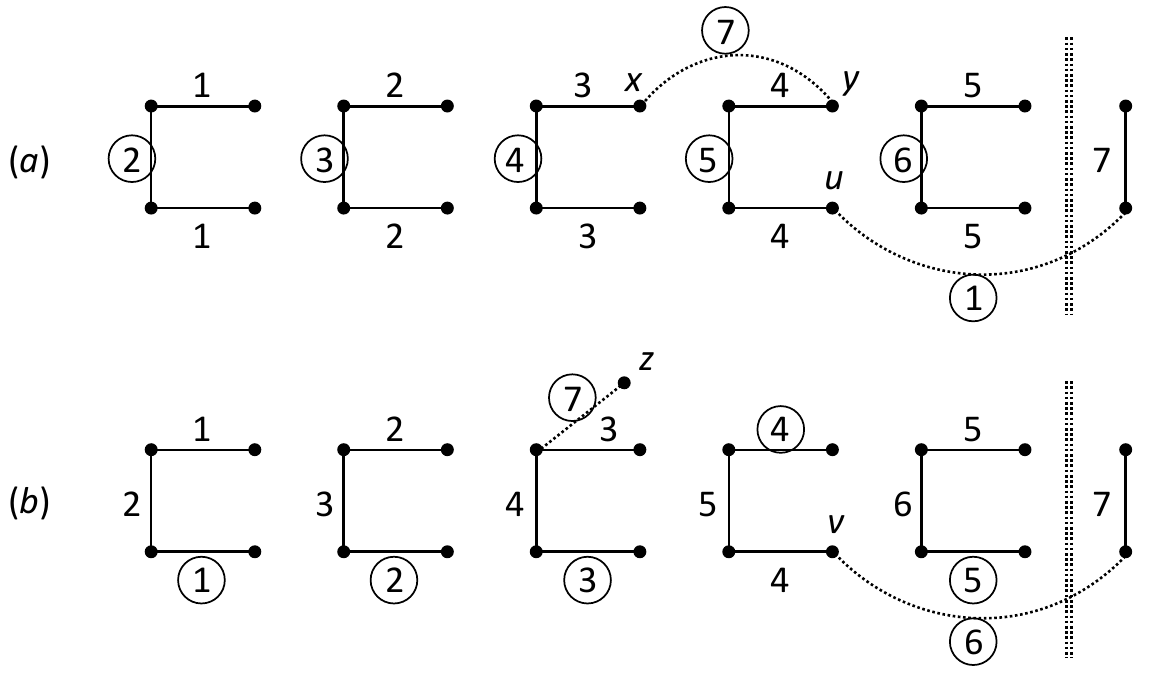}}
\caption{\label{fig_outside}Proof of Lemma~\ref{lemma_outside}. Here, $t=6$ and $i=7$, and in case (\emph{b}), $j=4$.}
\end{figure}

\begin{proof}
In case (\emph{a}), take an element $u\in S_i$. Note that $u\notin\{x,y\}$ by assumption. Proceed as in Figure~\ref{fig_outside}(\emph{a}).

In case (\emph{b}), take an element $v\in T_i$. Note that $v\notin\{z\}\cup C_j$ by assumption. Proceed as in Figure~\ref{fig_outside}(\emph{b}).
\end{proof}

Recall that $|H|\ge n/5$. Fix an index $i\in H$, and consider the set $L_i = K_i\setminus (B'\cup U_i)$. By Lemma~\ref{lemma_outside}(\emph{a}), each element of $L_i$ must be $i$-equivalent to a \emph{different} element of $B\cup S_i$ (by transitivity of $\eqv{i}$). Hence, there are at most two elements of $L_i$ that are $i$-equivalent to elements of $S_i$, and at most four more that are $i$-equivalent to $a_j$ or $b_j$ in an $i$-tainted component $C_j$. All the remaining elements of $L_i$ must be $i$-equivalent to elements of $\Cright$, by Lemma~\ref{lemma_outside}(\emph{b}).

Hence, let us $i$-charge the elements of $K_i\setminus U_i$ to components according to the following charging scheme (which is similar to Charging Scheme 1):

\begin{charging}
Let $i\in H$. Consider an element $x\in K_i\setminus U_i$. If $x\in B'$, then $x$ is $i$-charged to the component it belongs to. Otherwise, if $x$ is  $i$-equivalent to an element of $S_i$ or to $a_j$ or $b_j$ where component $C_j$ is $i$-tainted, then $x$ is not $i$-charged. Otherwise, $x$ must be $i$-equivalent to an element $y\in \Cright$; then we charge $x$ to the component that contains $y$.
\end{charging}

\begin{lemma}\label{lemma_cs3}
In Charging Scheme 3, no component receives more than four $i$-charges.
\end{lemma}

\begin{proof}
By the above considerations, left-side components only receive charges from their own elements, and right-side components can receive at most two outside charges.
\end{proof}

We have $|K_i\setminus U_i| \ge |K_i|- 4$, and there are at most six elements of this set that are not charged. Hence, there are at least $n/5 + c - 10$ components that received at least four charges. Out of them, at least $c-10$ are right-side components.

Let us apply this charging for all $i\in H$. By the pigeonhole principle, there must be a ``lucky" right-side component $C_{j^\star}$ that receives four $i$-charges for all $i\in H'$, for some subset $H'\subseteq H$ of size $|H'| = ((c-10)n/5) / (4n/5) = (c-10)/4$. For each index $k\in H'$, let $W_k$ be the set of two elements not in $C_{j^*}$ that were $k$-charged to $C_{j^*}$. Note that each $W_k$ is disjoint from $B'$, since elements of $B'$ were charged to their own components. Furthermore, $W_k$ is disjoint from $U_k$ for each $k$ (since the elements of $U_k$ were not used).

Let $k_1, k_2$ be a pair of distinct indices in $H'$. We would like to choose four distinct elements $w, x, y, z$, with $C_{j^\star} = \{w, x\}$ and $y,z \in W_{k_1} \cup W_{k_2}$, such that $w\eqv{k_1} y$ and $x\eqv{k_2} z$. If such a choice is not possible, then call the pair $k_1$, $k_2$ ``conflicting".

\begin{lemma}\label{lemma_IS}
There exists a subset $H''\subset H'$, of size $|H''|=c/16$, such that no two indices in $H''$ are conflicting.
\end{lemma}

\begin{proof}
The \emph{only} way for $k_1,k_2$ to be conflicting is to have  $C_{j^\star} = \{w, x\}$, $W_{k_1} = W_{k_2} = \{y,z\}$, $w\eqv{k_1} y$, $x\eqv{k_1} z$, $w\eqv{k_2} z$, $x\eqv{k_2} y$. Therefore, if we define an undirected graph having $H'$ as vertex set, and having an edge $k_1k_2$ whenever $k_1, k_2$ are conflicting, this graph cannot have an odd cycle, and therefore it is bipartite, and hence it has an independent set of size at least $|H'|/2 = (c-10)/8$. Out of this independent set, we select an arbitrary subset $H''$ of size $c/16$.
\end{proof}

Let $H''\subset H'$ be as in Lemma~\ref{lemma_IS}, and let $W=\bigcup_{k\in H''} W_k$. We do not have a good handle on the size of $W$; it could be anything in the range $2\le |W|\le 2|H''|$.

Let $k_1, k_2$ be a pair of distinct indices in $H''$. We already know that $k_1$, $k_2$ are not conflicting, and that $W_{k_1} \cap U_{k_1} = W_{k_2} \cap U_{k_2} = \emptyset$. We would also like to have $W_{k_1} \cap U_{k_2} = W_{k_2} \cap U_{k_1} = \emptyset$. If that is the case, call the pair $k_1$, $k_2$ ``compatible".

\begin{lemma}\label{lemma_compatible}
Let $c=5000$. Then there exists a compatible pair of distinct indices $k_1,k_2\in H''$.
\end{lemma}

\begin{proof}
Call an element of $W$ ``popular" if it appears in at least $\sqrt c$ sets $W_k$, $k\in H''$. Let $W_{\mathrm p} \subseteq W$ be the set of popular elements; we have $|W_{\mathrm p}| \le 2|H''|/\sqrt c \le \sqrt c/8$.

Since each element can appear in at most two sets $U_k$ (at most one $S_k$ and at most one $T_k$), there are at least 
\begin{equation*}
|H''| - 2|W_{\mathrm p}|\ge c/16 - \sqrt c/4\ge 1
\end{equation*}
indices $k_1 \in H''$ for which $U_{k_1}$ contains no popular element.

Pick one such index $k_1$. The elements of $U_{k_1}$ appear in at most $4\sqrt c$ sets $W_{k_2}$, and the elements of $W_{k_1}$ appear in at most $4$ other sets $U_{k_2}$. That leaves us with at least
\begin{equation*}
c/16 - 4\sqrt c-4 \ge 1
\end{equation*}
choices for $k_2\in H''$ that make the pair $k_1$, $k_2$ compatible.
\end{proof}

We are almost done:

\begin{lemma}\label{lemma_jstar}
If there exist two distinct elements $x\eqv{j^\star} y$, both outside 
\begin{equation*}
Y=\Cright \cup W\cup \bigl(\bigcup_{k\in H''} U_k \bigr),
\end{equation*}
then we are done.
\end{lemma}

\begin{proof}
Complications arise only when $\{x,y\} \subset \Cleft$.

Suppose first that $x$ and $y$ belong to the same left-side component. By Lemma~\ref{lemma_compatible}, let $k_1, k_2 \in H''$ be a compatible pair of distinct indices. Invoke Lemma~\ref{lemma_2_heavy} with $i=k_1$, $j=k_2$. We obtain four distinct elements $v_1\eqv{1} w_1$, $v_2\eqv{t} w_2$, with $v_1, v_2 \in C_{k_1} \cup C_{k_2}$ with $w_1, w_2 \in U_{k_1} \cup U_{k_2}$. If we are unlucky and $\{w_1, w_2\} = \{c_j, d_j\}$ for some index $j$, then we invoke the ``furthermore'' clause of Lemma~\ref{lemma_2_heavy} with $\{q,r\} = \{w_1,w_2\}$, and we proceed as in Figure~\ref{fig_jstar}(\emph{a}).

\begin{figure}
\centerline{\includegraphics{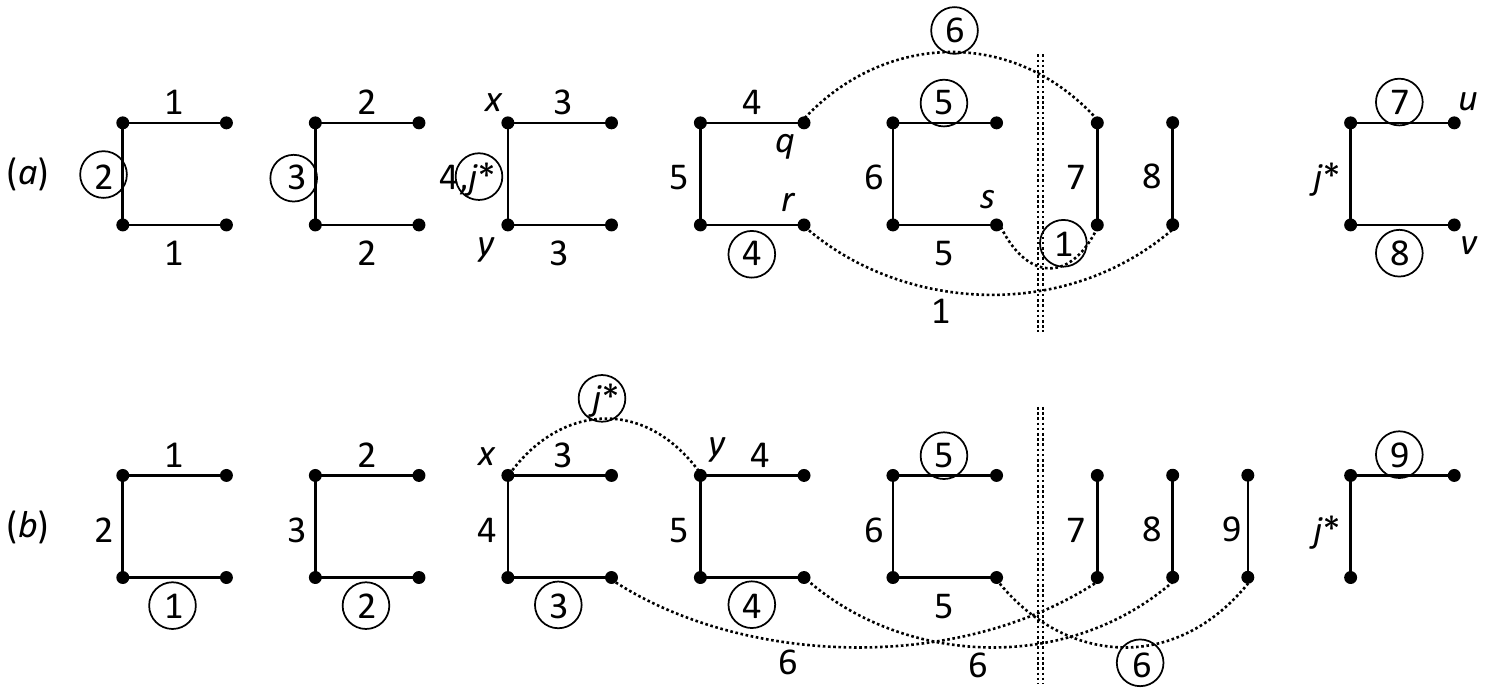}}
\caption{\label{fig_jstar}Proof of Lemma~\ref{lemma_jstar}. Here $t=6$. In case (\emph{a}) we have $\{k_1,k_2\} = \{7,8\}$. Since the indices $k_1,k_2$ are nonconflicting and compatible, the elements $q,s,u,v$ are all distinct. In case (\emph{b}) we have $H'' \supset \{7,8,9\}$ and $k=9$.}
\end{figure}

Now suppose $x$ and $y$ belong to different left-side components $C_i$ and $C_j$. Suppose without loss of generality that $x\in\{a_i,c_i\}$ and $y\in\{a_j,c_j\}$. Then we take a $k\in H''$ such that $d_i, d_j \notin T_k$. Such a $k$ must exist, since $|H''|\ge 3$. Then we proceed as in Figure~\ref{fig_jstar}(\emph{b}).
\end{proof}

But a pair $x$, $y$ as in Lemma~\ref{lemma_jstar} must exist, since otherwise, every element of $K_{j^\star}$ would either belong to $Y$ or be $j^\star$-equivalent to a different element of $Y$. Observe that $|Y| \le 2(4/5)n + c/8 + c/4$. That accounts for only $2|Y| \le (3+1/5)n + 3c/4$ elements of $K_{j^\star}$, which is not enough. This concludes the proof of Theorem~\ref{thm_bd}. \hfill $\qed$

\begin{remark}
The main difference between Charging Scheme 2 on the one hand, and Charging Schemes 1 and 3 on the other hand, is the way they handle the elements of $B' \setminus B$. In Charging Schemes 1 and 3 these elements are automatically charged to the component they belong to, whereas in Charging Scheme 2 they are not. If we modified Charging Scheme 3 to be like Charging Scheme 2 in this respect, then no left-side component would receive more than two $i$-charges, improving Lemma~\ref{lemma_cs3}. However, we then run into trouble since the sets $W_k$ might intersect $B'\setminus B$.
\end{remark}

\paragraph{Acknowledgements} Thanks to L. \v S. (Yehuda) Grinblat for suggesting us to look at this problem and for helpful discussions. Special thanks to the referees for reading the paper carefully and providing detailed suggestions. Thanks also to Anat Paskin-Cherniavsky for helpful discussions.

\end{document}